\documentclass[12pt]{amsart}

\usepackage{amssymb}

\usepackage{enumitem}

\usepackage{graphicx}

\makeatletter
\@namedef{subjclassname@2020}{%
	\textup{2020} Mathematics Subject Classification}
\makeatother

\usepackage[T1]{fontenc}

\usepackage{multirow}

\newtheorem{theorem}{Theorem}[section]
\newtheorem{corollary}[theorem]{Corollary}
\newtheorem{lemma}[theorem]{Lemma}
\newtheorem{proposition}[theorem]{Proposition}

\theoremstyle{definition}
\newtheorem{definition}[theorem]{Definition}
\newtheorem{remark}[theorem]{Remark}

\numberwithin{equation}{section}

\begin{document}
\title[ Extreme contractions on finite-dimensional Banach spaces]{ Extreme contractions on finite-dimensional Banach spaces}
\author[Sain, Sohel and  Paul  ]{Debmalya Sain, Shamim Sohel and Kallol Paul }

\newcommand{\acr}{\newline\indent}
\address[Sain]{Departmento de Analisis Matematico\\ Universidad de Granada\\ Spain. }
\email{saindebmalya@gmail.com}
\address[Sohel]{Department of Mathematics\\ Jadavpur University\\ Kolkata 700032\\ West Bengal\\ INDIA}
\email{shamimsohel11@gmail.com}
\address[Paul]{Department of Mathematics\\ Jadavpur University\\ Kolkata 700032\\ West Bengal\\ INDIA}
\email{kalloldada@gmail.com, kallol.paul@jadavpuruniversity.in}

\begin{abstract}
	We study extreme contractions in the setting of finite-dimensional polyhedral Banach spaces. Motivated by the famous Krein-Milman Theorem, we prove that a \emph{rank one} norm one linear operator between such spaces can be expressed as a convex combination of \emph{rank one} extreme contractions, whenever the domain is two-dimensional. We establish that the same result holds true in the space of all linear operators from $\ell_{\infty}^n(\mathbb{C}) $ to $ \ell_1^n (\mathbb{C}). $ Furthermore, we present a geometric characterization of extreme contractions between finite-dimensional polyhedral Banach spaces. 
       
\end{abstract}


\maketitle

\section{Introduction.}

Identifying the extreme contractions is a classical problem in the geometry of Banach spaces.  Although extreme contractions between \emph{Hilbert spaces} were completely characterized by Kadison in his seminal work \cite{K2}, the problem is far from being well-understood in the framework of \emph{Banach spaces}, after all this time. Even for linear operators between \emph{finite-dimensional Banach spaces}, the study of extreme contractions is very much complicated and very little is known from a general perspective. One may explore \cite{G1,G2,I,K1,K2,L1, L2, L3, LP,MPD,RRBS,SRP,S1,S2} and the references therein, to gain more information in this context. Very recently, some initial progress has been made in \cite{SPM}, where extreme contractions were completely classified for \emph{two-dimensional real strictly convex and smooth Banach spaces}. In this article, we continue the advancement by characterizing the extreme contractions between \emph{finite-dimensional polyhedral Banach spaces}. The readers are invited to contrast the methods employed in this article with those in \cite{SPM}, to verify our claim that the presence  (or absence) of geometric features such as strict convexity and smoothness play a central role in determining the extremal structure of the unit ball of operator spaces. We also discuss a possible strengthening of the Krein-Milman Theorem, for \emph{rank one linear operators of unit norm} between certain finite-dimensional real or complex Banach spaces, including the polyhedral spaces.\\

We use the symbols $ \mathbb{X}, \mathbb{Y} $ to denote Banach spaces.  All the Banach spaces considered in this article will be real, unless we say the contrary.
Let $ B_{\mathbb{X}}= \{ x \in \mathbb{X}: \|x\| \leq 1\} $  and $ S_{\mathbb{X}}= \{ x \in \mathbb{X}: \|x\| = 1\} $ denote  the unit ball and the unit sphere of $\mathbb{X},$ respectively. The convex hull of a non-empty set $ S  \subset \mathbb{X}$  is the intersection of all convex sets in $\mathbb{X}$ containing $S$ and it is denoted by $Co(S).$ For any $\widetilde{x}=(x_1, x_2, \ldots, x_n)\in \mathbb{R}^n,$ we use the notation $ \widetilde{x} \geq 0 $ if each $ x_i \geq 0$ and $\widetilde{x} \leq 0$ if each $x_i \leq 0.$ Let $ \mathbb{L}(\mathbb{X}, \mathbb{Y})$ denote the Banach space of all bounded linear operators from $\mathbb{X}$ to $\mathbb{Y},$ endowed with the usual operator norm. For a  non-empty convex subset $ A $ of $  \mathbb{X},$ an element $ z \in A$ is said to be an extreme point of $A$ if $z = (1-t)x + ty,$ for some $t \in (0, 1)$ and some $ x, y \in A,$ implies that $ x = y = z.$ 
 The collection of all extreme points of  $A$ is denoted by $Ext(A).$  An operator $T \in \mathbb{L}(\mathbb{X}, \mathbb{Y})$ is said to be an extreme contraction if $T$ is an extreme point of $ B_{\mathbb{L}(\mathbb{X}, \mathbb{Y})}.$  It is trivial to observe that if $T$ is an extreme contraction then $\|T\|=1.$ Given any $ T \in \mathbb{L}(\mathbb{X}, \mathbb{Y}), $ let $M_T$ denote the norm attainment set of $T,$ i.e, $M_T=\{ x \in S_{\mathbb{X}}: \|Tx\|=\|T\|\}.$ The kernel of $T,$  denoted by $ Ker~T,$ is defined as $ Ker~T  = \{ x \in \mathbb{X} : Tx = 0 \}.$  A finite-dimensional Banach space $\mathbb{X}$ is said to be polyhedral if $Ext(B_{\mathbb{X}})$ is finite. In other words, $\mathbb{X}$ is a  polyhedral Banach space if $ B_{\mathbb{X}} $ is a polyhedron in $ \mathbb{R}^n, $ for some natural number $ n \geq 2. $ With regards to the geometric structure of such a space, let us now recall the following standard definitions. For a detailed exposition on polyhedrons and their properties, we refer to \cite{A, SPBB}.

\begin{definition}
	Let $\mathbb{X}$ be a finite-dimensional  Banach space. A \emph{polyhedron} $P$ is a non-empty compact subset of $\mathbb{X},$ which is the intersection of finitely many closed half-spaces of $\mathbb{X},$ i.e., $P= \cap_{i=1}^rM_i,$ where $M_i$ are closed half-spaces in $\mathbb{X}$ and $r \in \mathbb{N}.$ The dimension $dim(P)$ of the polyhedron $P$ is defined as the dimension of the subspace generated by the differences $v-w$ of vectors $v, w \in P.$ 
\end{definition}

\begin{definition}
	A polyhedron $Q$ is said to be a \emph{face} of the polyhedron $P$ if either $Q=P$ or if we can write $Q= P \cap \delta M,$ where $ M$ is a closed half-space in $\mathbb{X}$ containing $P$ and $\delta M$ denotes the boundary of $M.$ If $dim(Q)=i,$ then $Q$ is called an \emph{$i$-face} of $P.$ $(n-1)$- faces of $P$ are called \emph{facets} of $P$ and $1$-faces of $P$ are called \emph{edges} of $P.$
\end{definition}

We also  require the notion of Birkhoff-James orthogonality \cite{B, J} in a Banach space. 	For $x, y \in \mathbb{X}, $ we say that $ x $ is Birkhoff-James orthogonal to $ y, $ written as $ x \perp_B y, $ if $ \| x+\lambda y \| \geq \| x \| $ for all $ \lambda \in \mathbb{R}. $ For any $x \in \mathbb{X},$ $x ^\perp$ is defined as the collection of all $y \in \mathbb{X}$ such that $ x \perp_B y.$ Note that $x^\perp$ is a non-trivial subset of $\mathbb{X}$ and always contains a subspace of co-dimension $1$ by virtue of the Hahn-Banach Theorem. The next definition is motivated by the famous Krein-Milman Theorem, whose Banach space version states that if $K$ is a non-empty compact convex subset of a Banach space $\mathbb{X}$ then $ K = \overline{Co(Ext(K))},$  and therefore, in particular, Ext$(K) \neq \emptyset. $

\begin{definition}
	Let $\mathbb{X},\mathbb{Y}$ be Banach spaces and let $r \in \mathbb{N}.$ We say that the pair $(\mathbb{X}, \mathbb{Y})$ satisfies the \emph{rank invariant Krein-Milman Property of order $r$} if every norm one linear operator in $\mathbb{L}(\mathbb{X}, \mathbb{Y})$ having  rank $r$  can be expressed as a convex combination of extreme contractions $\mathbb{L}(\mathbb{X}, \mathbb{Y})$ having rank $r.$
\end{definition}

 In case of finite-dimensional polyhedral Banach spaces, to have a better understanding of the rank invariant Krein-Milman Property of order $1,$ the following definition turns out to be very useful. 
	\begin{definition}
		Let $\mathbb{X}$ and $\mathbb{Y}$ be finite-dimensional polyhedral Banach spaces, where $\dim(\mathbb{X})=n.$ We say that an $ n \times n$ matrix $ A=(a_{ij})_{ 1 \leq i,j \leq n}$ is a \emph{Krein-Milman companion matrix} if there exist	
		\begin{itemize}
				\item[(i)] rank $1$ extreme contractions $T_1, T_2, \ldots, T_n $ in $\mathbb{L}(\mathbb{X}, \mathbb{Y}),$  and 
				\item[(ii)]  $ v \in  \bigcap_{j=1}^n (M_{T_j} \cap Ext(B_{\mathbb{X}}))$ satisfying $ T_1v= T_2v= \ldots= T_nv$ 
					\end{itemize}
				such that the following holds:  $ T_jz_i= a_{ij}T_jv,$ and $ a_{1j}=1,$  for $ 2 \leq i \leq n,~  1 \leq j \leq n, $ where 
				$z_2, z_3, \ldots, z_n \in v^\perp \cap S_{\mathbb{X}}$ are linearly independent. 
	
	\end{definition}

	The reader should note that in the above definition, the conditions (i) and (ii) are ensured by virtue of \cite[Th.2.1]{RS} and Theorem \ref{rank1} of the present article. Let $ \mathbb{H} $ be a Hilbert space of dimension  $n >1. $ $ T \in \mathbb{L}(\mathbb{H}) $ is an extreme contraction if and only if $T$ is an isometry. In particular, for the $n$-dimensional Euclidean space $ \mathbb{H} $, the pair $ (\mathbb{H},\mathbb{H}) $  possesses the rank invariant Krein-Milman Property of order $ r$ if and only if $ r=n. $\\	
	
	The present article is divided into four sections including the introductory one. In the next section,  we focus exclusively on \emph{rank one} operators. In particular, we characterize the \emph{rank one} extreme contractions on finite-dimensional polyhedral Banach spaces. As an application of this result, we obtain a sufficient condition for the rank invariant Krein-Milman Property of order $1,$ in the said setting. In the succeeding section, we illustrate that such a strengthening of the Krein-Milman Property for rank $ 1 $ operators is not limited to real Banach spaces. In particular, we prove that the pair $(\ell_{\infty}^n, \ell_{1}^n)$ satisfies the rank invariant Krein-Milman Property of order $1,$ over the complex field $\mathbb{C}.$ 
	In the final section, we characterize the extreme contractions in $\mathbb{L}(\mathbb{X}, \mathbb{Y}),$ where $\mathbb{X}, \mathbb{Y}$ are two finite-dimensional polyhedral Banach spaces.
	

 \section{Main Results.}
 
 \section*{Section: I}
 
 We  begin this section with the following two basic observations which will come handy in our discussions. The proofs are omitted, as they are quite elementary.
 
 \begin{proposition}\label{prop1}
 	Let $ \mathbb{X} $ be an $n$-dimensional polyhedral Banach space.  Then every facet $ F$ of $B_{\mathbb{X}}$ contains  $n$ linearly independent extreme points of $ B_{\mathbb{X}}.$ 
 \end{proposition}

 \begin{proposition}\label{prop2}
 	Let $ \mathbb{X} $ be an $n$-dimensional polyhedral Banach space. Then a face $F(\neq B_{\mathbb{X}})$ of the polyhedron $ B_{\mathbb{X}}$  is a facet  if $F$ contains $n$ linearly independent points.
 \end{proposition}


In the following theorem, we characterize the \emph{rank one} extreme contractions between finite-dimensional polyhedral Banach spaces.

\begin{theorem}\label{rank1}
	Let $\mathbb{X}$ and $\mathbb{Y}$ be finite-dimensional polyhedral Banach spaces, where $dim(\mathbb{X})=n.$ Let $T \in S_{\mathbb{L}(\mathbb{X}, \mathbb{Y})}$ be such that rank $ T= 1.$ Then $T$ is an extreme contraction if and only if there exists a facet $F$ of the polyhedron $B_{\mathbb{X}}$
	  such that $  M_{T}= F \cup (-F) $ and $ Tx \in Ext(B_{\mathbb{Y}}),$ for any $x \in F. $ 
\end{theorem}

\begin{proof}
	We first prove the sufficient part of the theorem. Let $T= \frac{1}{2} (T_1 + T_2),$ where $T_1, T_2 \in S_{\mathbb{L}(\mathbb{X}, \mathbb{Y})}.$ By the hypothesis, there exists a facet $F$ of the polyhedron $ B_{\mathbb{X}}$ such that $ Tx \in Ext(B_{\mathbb{Y}})$, for any  $ x \in F.$ Following Proposition \ref{prop1}, there exist $n$  linearly independent extreme points $v_1, v_2, \ldots, v_n \in F$ such that $Tv_i \in Ext(B_{\mathbb{Y}}),$ for all $ i \in \{ 1, 2, \ldots, n\}.$ Then for any $ 1 \leq i \leq n, $ $Tv_i= \frac{1}{2} (T_1v_i+ T_2v_i),$ which implies that $Tv_i= T_1v_i= T_2v_i.$ Therefore, $T= T_1= T_2$ and consequently, $T$ is an extreme contraction. 
	
	We next prove the necessary part of the theorem.  Since rank $T=1,$ suppose that $T(\mathbb{X}) \cap S_{\mathbb{Y}}= \{\pm w\}$. Since $T\in S_{\mathbb{L}(\mathbb{X}, \mathbb{Y})}$ is an extreme contraction, it follows from \cite[Th. 2.2]{SRP} that  $ span(M_T \cap Ext(B_\mathbb{X}))= \mathbb{X}.$  Let $v_1, v_2, \ldots, v_n$ be $n$ linearly independent extreme points of $B_{\mathbb{X}}$ such that $ Tv_i=w,$ for all $ i \in \{1, 2, \ldots, n\}.$ For any $ x \in Co(\{v_1, v_2, \ldots, v_n\}) \cap S_\mathbb{X},$ where $x=c_1 v_1 + c_2v_2+ \ldots+ c_n v_n, $
	\[ Tx= T(c_1 v_1 + c_2v_2+ \ldots+ c_n v_n)= (c_1 +c_2+\ldots+ c_n) w=w.
	\]
	This implies that $x \in M_T$ and therefore, $Co(\{v_1, v_2, \ldots, v_n \}) \cap S_\mathbb{X} \subset M_T.$ Let us consider the set  $M= \{ (d_1v_1+ d_2v_2+ \ldots+ d_nv_n) \in \mathbb{X}: d_1+ d_2+ \ldots+ d_n \leq 1\}$ and let
	  $F= S_{\mathbb{X}} \cap \delta M,$ where $\delta M$ is the boundary of $M.$ From Proposition \ref{prop2}, $ F$ is a facet of $ B_{\mathbb{X}}.$ Observe that $Co(\{v_1, v_2, \ldots, v_n \})  \cap S_\mathbb{X} \subset F$ and from Proposition \ref{prop2}, it follows that $ F \subset M_T.$  Since $ T$ is rank $1$, we conclude that $ M_T=  F \cup (-F).$ 
	   We next claim that $Tx=w \in Ext(B_\mathbb{Y}),$ for any $x \in F.$ Let $ w_1, w_2 \in S_{\mathbb{Y}}$ such that $w=\frac{1}{2} (w_1+ w_2).$  Define $T_1, T_2 \in \mathbb{L}(\mathbb{X}, \mathbb{Y})$ such that $T_1v_i=w_1$ and $T_2v_i=w_2,$ for all $ i \in \{ 1, 2, \ldots, n\}.$ Clearly,  $T= \frac{1}{2} (T_1+T_2).$  For any $\widetilde{u} \in S_{\mathbb{X}} \setminus \pm F,$ it is easy  to observe that $ \widetilde{u} \in \{ (d_1v_1+ d_2v_2+ \ldots + d_nv_n) \in \mathbb{X} : -1 < d_1+ d_2+ \ldots + d_n <1\}$ and therefore, $ T_1, T_2 \in S_{\mathbb{L}(\mathbb{X}, \mathbb{Y})}.$  Since $T$ is an extreme contraction, $T= T_1= T_2.$ This proves that $w= w_1= w_2$ and hence $ w$ is an extreme point of $B_{\mathbb{Y}}.$
\end{proof} 

The above characterization allows us to explicitly compute the number of rank $1$ extreme contractions between any pair of polyhedral Banach spaces, in the finite-dimensional context. This is presented in the following corollary.

\begin{corollary}
	Let $\mathbb{X}$ and  $\mathbb{Y}$ be finite-dimensional polyhedral Banach spaces. Suppose that the number of facets of $ B_{\mathbb{X}}$ is $2r$ and the number of extreme points of  $B_{\mathbb{Y}} $ is $ 2s.$ Then the number of rank $1$ extreme contractions in $ \mathbb{L}(\mathbb{X}, \mathbb{Y})$ is $2rs.$
\end{corollary}

\begin{proof}
	Let $ \pm F_1, \pm F_2, \ldots, \pm F_r$ be the facets of $ B_{\mathbb{X}}$ and  let $ Ext(B_{\mathbb{Y}})= \{ \pm w_1, \pm w_2, \ldots,$ $ \pm w_s\}.$  It follows from Theorem \ref{rank1} that any rank $1$ extreme contraction  is of the form   $ T_{ij}(x)=  w_j$  or  $ T_{ij}(x)=  -w_j, $ for any $x \in F_i,$ where $ 1 \leq i \leq r, ~1 \leq j \leq s.$ Therefore, the number of rank $1$ extreme contractions in $ \mathbb{L}(\mathbb{X}, \mathbb{Y})$ is $2rs.$
\end{proof}

In the two-dimensional case, of course, the number of facets of $ B_{\mathbb{X}}$ is equal to the number of extreme points of $ B_{\mathbb{X}}$, and therefore, the following simple formulation holds true:\\
\emph{Whenever the domain is two-dimensional, the number of rank one extreme contractions between two polyhedral Banach spaces is equal to half the product of the number of extreme points of the unit balls of the two spaces.}\\

In addition, we should also observe an important fact regarding the extremal structure of operator spaces, in the polyhedral case. It is elementary to see that purely from a cardinality point of view, the extremal structure of the unit ball of a given Banach space does not determine the geometry of space. In other words, two polyhedral Banach spaces of the same dimension, which are not isometric, may  have the same number of extreme points of their respective unit balls. In this context, the above corollary ensures that the number of rank one extreme contractions between finite-dimensional polyhedral spaces is dependent \emph{only} on the extremal structures of the two spaces, \emph{only} from the point of view of cardinality. The isometric structures of the two spaces do not apparently play any role whatsoever, as far as rank one extreme contractions are concerned.

Our next goal is to present a sufficient condition for the rank invariant Krein-Milman Property of order $1$, in case of a pair of finite-dimensional polyhedral Banach spaces. We require the following well-known result for this purpose.

\begin{lemma}\cite{GKT}{(Farkas' Lemma)}\label{farkas}
	 Let  $A \in \mathbb {R} ^{m\times n}$ and let $ b \in \mathbb {R}^m. $ Then exactly one of the following two assertions is true:
	
\item(i)	There exists an $x \in \mathbb {R} ^{n}$ such that  $Ax =b $  and  $x  \geq 0.$

\item(ii)	There exists a $y \in \mathbb {R} ^{m}  $ such that $A ^t y \geq 0$ and $b ^t y <0.$
\end{lemma}

Using this result, we have the following sufficient condition for rank invariant Krein-Milman Property of order $1.$
 
\begin{theorem}\label{companion}
		Let $\mathbb{X}$ and $\mathbb{Y}$ be  finite-dimensional polyhedral Banach spaces, where $dim(\mathbb{X})=n.$ The pair $(\mathbb{X}, \mathbb{Y})$ satisfies the rank invariant Krein-Milman Property of order $1$ if for any $ \widetilde{y}=(y_1, y_2, \ldots, y_n)\in \mathbb{R}^n$ satisfying $ y_1 < 0 $ and for any Krein-Milman companion matrix $ A,$  $A^t\widetilde{y}^t \ngeqq 0.$ 
\end{theorem}

\begin{proof}
		Let $T \in S_{\mathbb{L}(\mathbb{X}, \mathbb{Y})}$  be a rank $1$ linear operator and  let $v \in M_T \cap Ext(B_{\mathbb{X}}) $.  Suppose that $ Tv=w= c_1w_1+ c_2w_2+ \ldots+ c_rw_r, $ where $w_1, w_2, \ldots, w_r$ are extreme points of the polyhedron $B_{\mathbb{Y}}$ and $ c_1, c_2, \ldots, c_r \in [0, 1]$ such that $ \sum_{i=1}^{r} c_i =1.$ Let $ \{z_2, z_3, \ldots, z_n \} \subset Ker T \cap S_{\mathbb{X}}$ be a linearly independent set.  It is easy to observe that $ z_2, z_3, \ldots, z_n \in v^\perp \cap S_{\mathbb{X}}.$ From Step $1$ in the proof of  \cite[Th. 2.1]{RS}, we conclude that there are at least  $n$-number of facets of $B_{\mathbb{X}}$ containing $v.$ Suppose that $F_1, F_2, \ldots, F_n$ are the facets of $ B_{\mathbb{X}}$ such that $ v \in \cap_{i=1}^nF_i  .$   For  $ 1\leq i \leq r, 1\leq j \leq n,$ define $T_{ij} \in \mathbb{L} (\mathbb{X}, \mathbb{Y})$ as  $ T_{ij}(x)= w_i,$ for any $ x \in F_j.$ From Theorem \ref{rank1}, it follows that $T_{ij}$  are extreme contractions of 
		rank $1$, for all $1 \leq i \leq r, 1 \leq j \leq n.$  Let $ T_{ij}z_k= d_{ij}^kw_i,$  where $d_{ij}^k \in [-1,1],$ for any $ 1 \leq i \leq r, 1\leq j \leq n, 2 \leq k \leq n .$  We claim that $T$ is a convex combination of $rn$ extreme contractions $T_{ij}, 1 \leq i \leq r, 1 \leq j \leq n.$
 To establish our claim, we consider the following $r$ number of systems of linear equations in $nr$ variables $\alpha_{ij}$ 
 	\begin{eqnarray}\label{eqn}
 		\begin{aligned}
 				&&\alpha_{i1} + \alpha_{i2} + \ldots + \alpha_{in} = 1\\
 			d_{i1}^2 \alpha_{i1}&+& d_{i2}^2 \alpha_{i2} + \ldots + d_{in}^2 \alpha_{in} = 0\\
 			d_{i1}^3\alpha_{i1}&+& d_{i2}^3\alpha_{i2} + \ldots + d_{in}^3\alpha_{in} = 0\\
 			&.&\\
 			&.&\\
 			&.&\\
 			d_{i1}^n\alpha_{i1} &+& d_{i2}^n\alpha_{i2} + \ldots + d_{in}^n\alpha_{in} = 0,
 		\end{aligned}
  \end{eqnarray}
for any $1 \leq i \leq r.$
For each $k=1,2,\ldots,r,$ consider the matrix  $ P_k= (p_{ij}^k)_{1 \leq i,j \leq n},$ where  $ p_{1j}^k=1,$ for $  1 \leq j \leq n $ and $p_{ij}^k= d_{kj}^i,$ for  $ 2 \leq i \leq n, 1\leq j \leq n.$  Observe that  $ T_{kj}z_i= d_{kj}^i T_{kj}v $ and so  clearly for each $ k,~ 1\leq k \leq r,$ $ P_k$ is a Krein-Milman companion matrix.  From the hypothesis for any $ \widetilde{y}=(y_1, y_2, \ldots, y_n)\in \mathbb{R}^n,$ with $ y_1 <0,$ we have  $P_k^t\widetilde{y}^t \ngeqq 0.$ Therefore, by using  Lemma \ref{farkas}, the system of linear equations (\ref{eqn}) has a positive solution.  In other words, there exist  $nr$ real numbers $ \alpha _{ij} \in [0, 1]$ satisfying the r number of system of linear equations  (\ref{eqn}). 
Then  for any $1 \leq i \leq r,$ we have 
	\begin{eqnarray*}
	\begin{aligned}
		(\alpha_{i1}T_{i1}&+& \alpha_{i2}T_{i2} + \ldots + \alpha_{in} T_{in}) v= w_i\\
		(\alpha_{i1}T_{i1}&+& \alpha_{i2}T_{i2} + \ldots + \alpha_{in} T_{in}) z_2= 0\\
		(\alpha_{i1}T_{i1}&+& \alpha_{i2}T_{i2} + \ldots + \alpha_{in} T_{in}) z_3= 0\\
		&.&\\
		&.&\\
		(\alpha_{i1}T_{i1}&+& \alpha_{i2}T_{i2} + \ldots + \alpha_{in} T_{in}) z_n= 0.
	\end{aligned}
\end{eqnarray*}
This implies
\begin{eqnarray*}
	Tv= w & = & c_1w_1+ c_2w_2+ \ldots+ c_rw_r = \sum_{i=1} ^r c_i \Big (\sum_{j=1}^n \alpha_{ij}T_{ij}v\Big )\\
	\mbox{and for each }  k,~ 2 \leq k \leq n, & & \\
Tz_k= 0 & = & c_1w_1+ c_2w_2+ \ldots+ c_rw_r = \sum_{i=1} ^r c_i \Big (\sum_{j}^n \alpha_{ij}T_{ij}z_k\Big ).
\end{eqnarray*} 
Therefore, 
\begin{eqnarray*}
	 T  =   \sum_{i=1}^r \sum_{j=1}^n c_i \alpha_{ij}T_{ij} 
\end{eqnarray*}
with 
\[\sum_{i=1}^{r} [c_i (\alpha_{i1}+  \alpha_{i2}+ \ldots+ \alpha_{in}) ]= c_1+ c_2 +\ldots+ c_r=1.\]
This completes the theorem. 
 \end{proof}

As a consequence of the previous theorem, we have the following nice outcome whenever the domain is two-dimensional.

\begin{corollary} \label{dim2}
	 	Let $\mathbb{X}$ be a two-dimensional polyhedral Banach space and let $\mathbb{Y}$ be a finite-dimensional polyhedral Banach space. Then the pair $(\mathbb{X}, \mathbb{Y})$ satisfies the rank invariant Krein-Milman Property of order $1$.
\end{corollary}
\begin{proof}
	
Let $ A =(a_{ij})_{1 \leq i,j \leq 2} $ be a Krein-Milman companion matrix. 
Then $ a_{11} = a_{12}=1 .$  Suppose that $ T_1, T_2$ are two distinct rank $1$ extreme contraction in $ \mathbb{L}(\mathbb{X}, \mathbb{Y})$  such that  $ v \in M_{T_1} \cap M_{T_2} \cap Ext( B_{\mathbb{X}})$ and $ T_1z = a_{21}T_1v $ and $ T_2z= a_{22}T_2v,$ for some $ z \in v^\perp \cap S_{\mathbb{X}}$. Moreover, $T_1v=T_2v.$ Let $ E_1, E_2$ be the edges of $B_{\mathbb{X}}$ such that $ v \in E_1 \cap E_2.$   From Theorem \ref{rank1}, it is easy to observe that $ M_{T_{1}}=  E_1 \cup (-E_1)$ and $  M_{T_{2}}= E_2 \cup (-E_2).$ Then there exist two linear functionals $f_1, f_2 \in S_{\mathbb{X}^*}$ such that $T_1(x)= f_1(x)T_1v$ and $T_2(x)= f_2(x) T_2(v).$ Accordingly, the space $\mathbb{X}$ is divided into three mutually disjoint subsets as $H_1^+=\{x \in \mathbb{X}: f_1(x)>0\},$ $H_1^-=\{ x \in \mathbb{X}: f_1(x)< 0\}$ and $H_1= \{ x\in \mathbb{X}: f(x)=0)\}.$ Similarly, for the functional $f_2,$ we  get another decomposition of $ \mathbb{X} $ into three mutually disjoint subsets $H_2^+, H_2^-, H_2,$ defined accordingly. Clearly, $f_1$ and $f_2$ are  the supporting functionals corresponding to the facets $E_1$ and $E_2$ of $B_{\mathbb{X}},$ respectively and so  $f_1, f_2 \in Ext(B_{\mathbb{X}^*})$. It is easy to observe that any supporting functional at $v$ can be written as a convex combination of $f_1$ and $f_2.$
 We claim that for any  $ z \in v^\perp \cap S_{\mathbb{X}}$, $ z \in ( H_1^{+} \cap H_2^{-} )\cup ( H_1^{-} \cap H_2^{+}).$ Suppose on the contrary that $v \perp_B z$ but $z \in H_1^+ \cap H_2^+.$ Then $f_1(z), f_2(z)> 0.$ Since $v \perp_B z,$ there exists $g \in S_{\mathbb{X}^*}$ such that $g(v)=1$ and $g(z)=0.$ Now, $g= (1-t) f_1+ tf_2,$ where $t\in (0,1).$ Then $g(z)= (1-t) f_1(z)+ tf_2(z) \implies (1-t)f_1(z)+ t f_2(z)=0.$ That contradicts that $f_1(z), f_2(z) > 0.$ Thus, $z \notin H_1^+ \cap H_2^+.$ Similarly, we can show that $z \notin H_1^- \cap H_2^-.$ This establishes that  $v^\perp \subset  ( H_1^{+} \cap H_2^{-} )\cup ( H_1^{-} \cap H_2^{+}) \cup H_1 \cup H_2. $
 As a result, $f_1(z)f_2(z) \leq 0,$ and consequently $a_{21} a_{22} \leq 0.$
  Therefore, any Krein-Milman companion matrix $A$ is of the form 
$ 
\begin{pmatrix}
	1 & 1\\
	a_{21}& a_{22}
\end{pmatrix},$ where $a_{21} a_{22} \leq 0.$
Now for any $ \widetilde{y}= (y_1, y_2) \in \mathbb{R}^2$ such that $y_1 <0,$ 
\[
A^t\widetilde{y}^t=\begin{pmatrix}
	1 & a_{21}\\
	1& a_{22}
\end{pmatrix}
\begin{pmatrix}
	y_1\\
	y_2
\end{pmatrix}=
\begin{pmatrix}
	y_1 + a_{21}y_2\\
	y_1+ a_{22}y_2
\end{pmatrix}.
\]
Since $y_1 <0$ and $ a_{21} a_{22} \leq 0$ , it is evident that $ A^t\widetilde{y}^t \ngeqq 0.$ Therefore, the desired result follows immediately from Theorem \ref{companion}.
\end{proof}

\begin{remark}
It follows easily from   Corollary \ref{dim2} that the pair  $ (\ell_{\infty}^2, \ell_{\infty}^2) $ satisfies the rank invariant  Krein-Milman Property of order $1.$  However the pair fails to satisfy the rank invariant  Krein-Milman Property of order $2.$  The extreme contractions  (see \cite[Th. 2.3]{MPD})  of rank $2$ on $\ell_{\infty}^2$ are of the form : 
 \[
  \pm A_1 = \pm \begin{pmatrix}
 1 & 0 \\
 0 & 1
 \end{pmatrix},
\pm A_2= \pm \begin{pmatrix}
 	1 & 0 \\
 	0 & -1
 \end{pmatrix},
\pm A_3= \pm \begin{pmatrix}
	0 & 1 \\
	1 & 0
\end{pmatrix},
\pm A_4= \pm \begin{pmatrix}
	0 & 1 \\
	-1 & 0
\end{pmatrix}.
\]
Let $T=\begin{pmatrix}
	1 & 0\\
	\frac{1}{2} & \frac{1}{2}
\end{pmatrix}  \in S_{\ell_{\infty}^2}.$ Clearly, $rank ~T=2.$  We show  that $T$ cannot be written as a convex combination of rank  $2$ extreme contractions.  Suppose on the contrary that 
\[T =c_1 A_1 + c_2 (-A_1)+ c_3 A_2 + c_4 (-A_2)+ c_5 A_3 + c_6 (-A_3) + c_7 A_4 + c_8 (-A_4) ,\]
where $\sum_{i=1}^{8} c_i=1$ and $0 \leq c_i \leq 1.$ By  a straightforward calculation it can be shown that $c_1 -c_2+ c_3- c_4=1$ and $ c_5 - c_6 -c_7+c_8= \frac{1}{2}.$ This contradicts that  $\sum_{i=1}^{8} c_i=1$ and $0 \leq c_i \leq 1.$
 Therefore, $ (\ell_{\infty}^2, \ell_{\infty}^2) $ does not satisfy the rank invariant Krein-Milman Property of order $2.$ In this context, also observe that $ T=\frac{1}{2} A_1 + \frac{1}{2}B ,$ where $B= \begin{pmatrix}
 	1 & 0\\
 	1 & 0
 \end{pmatrix}$  is a rank $1$ extreme contraction. 
     
\end{remark}

\section*{Section-II}
The purpose of this section is to illustrate that the rank invariant Krein-Milman Property of order $ 1 $ may be satisfied in case of complex Banach spaces as well.
 Our principal objective is to show that   the pair  $(\ell_{\infty}^n, \ell_1^n)$ satisfies the said property.  We will use the notations $ \|.\|_\infty, \|.\|_1, \|.\|_{\infty, 1}, $ to denote the norms on $ \ell_{\infty}^n, \ell_{\infty}^n, \mathbb{L}(\ell_{\infty}^n, \ell_1^n), $ respectively. All the spaces considered in this section are over the complex field $\mathbb{C}.$

 We  would like to begin with a characterization of the rank $1$ operators and the rank $1$ extreme contractions in $\mathbb{L}(\ell_{\infty}^n, \ell_1^n).$ We require the following immediate description of extreme points of $ B_{\ell_{\infty}^n} $ and $ B_{\ell_1^n}. $

\begin{proposition}\label{proposition:extreme points}
	Let	$ \widetilde{u} = (u_1,u_2,\ldots,u_n) \in B_{\ell_{\infty}^n}. $ Then $ \widetilde{u} \in Ext(B_{\ell_{\infty}^n})$ if and only if $ |u_i|=1 $ for each $ i=1,2,\ldots,n. $ 
 On the other hand, $ \widetilde{w} = (w_1,w_2,\ldots,w_n) \in Ext(B_{\ell_1^n}) $ if and only if there exists $  i_0 \in \{1,2,\ldots,n\} $ such that $ |w_{i_{0}}| = 1 $ and $ w_{i} = 0 $ for each $ i \in \{1,2,\ldots,n\} \setminus \{ i_{0} \}. $ 
\end{proposition}

We next introduce a simple notation that will be particularly useful for our purpose. Let $ \widetilde{u} = (u_1,u_2,\ldots,u_n) \in Ext(B_{\ell_{\infty}^n}). $ We write $ \widetilde{u_1} = \widetilde{u}. $ For each $ i \in \{ 2,3,\ldots,n \}, $ we define $ \widetilde{u_i} = (u_1, u_2, \ldots \- ,  u_{i-1}, -u_i, u_{i+1}, $ $ \ldots, u_n). $ In other words, whenever $ i \in \{ 2,3,\ldots,n \}, $ $ \widetilde{u_i} $ differs from $ \widetilde{u} = \widetilde{u_1} $ only in the $ i- $th coordinate by a factor $ -1. $  It is easy to observe that $\{ \widetilde{u}, \widetilde{u_2}, \ldots, \widetilde{u_n} \}$ is a basis of $ \ell_{\infty}^n.$ We are now in a position to characterize the rank $ 1 $ operators on the unit sphere of $ \mathbb{L}(\ell_{\infty}^n, \ell_1^n). $

\begin{theorem}\label{theorem:rank $1$}
	Let $ T \in \mathbb{L}(\ell_{\infty}^n, \ell_1^n) $ be rank $ 1. $ Then $ \| T \|_{\infty,1} = 1 $ if and only if there exists $ \widetilde{u} = (u_1,u_2,\ldots,u_n) \in Ext(B_{\ell_{\infty}^n}), $ $ \widetilde{w} = (w_1,w_2,\ldots,w_n) \in S_{\ell_1^n} $ and $ \kappa_j \in [-1,1] $ such that for each $ j \in \{ 1,2,\ldots,n \}, $ $ T(\widetilde{u_j}) = \kappa_j \widetilde{w}, $ where $ \kappa_1 = 1 $ and $ \sum\limits_{j=1}^n \kappa_j \geq (n-2).  $
\end{theorem}

\begin{proof}
	Let us first prove the  necessary part of the theorem. Let $ T \in \mathbb{L}(\ell_{\infty}^n, \ell_1^n) $ be a rank $ 1 $ linear operator such that $ \| T \|_{\infty,1} = 1. $ Clearly, $ T $ attains norm at some extreme point $ \widetilde{u} = (u_1,u_2,\ldots,u_n) \in Ext(B_{\ell_{\infty}^n}). $ Let $ T(\widetilde{u}) = \widetilde{w} = (w_1,w_2,\ldots,w_n) \in S_{\ell_1^n}. $ Since $ T $ is rank $ 1, $ there exists $ \kappa_j \in \mathbb{C} $ such that $ T(\widetilde{u_j}) = \kappa_j \widetilde{w} $ for each $ j \in \{ 2,3,\ldots,n \}. $ We claim that each $ \kappa_j $ is real. To prove our claim, we first observe that $ \{ \widetilde{u_{j}}~:j=1,2,\ldots,n \} $ is a basis of $ \ell_{\infty}^n. $ Therefore, any $ \widetilde{z} = (z_1,z_2,\ldots,z_n) \in \ell_{\infty}^n  $ can be uniquely written as $  \widetilde{z} = \sum\limits_{j=1}^n \alpha_j \widetilde{u_j}, $ where $ \alpha_j $ are scalars. On explicit computation, we obtain the following expressions for $ \alpha_j : $ \\
	
	\noindent $ \alpha_1 = \frac{3-n}{2} \frac{z_1}{u_1} + \frac{1}{2} \sum\limits_{j=2}^n \frac{z_j}{u_j}, $\\
	
	\noindent $ \alpha_j = \frac{1}{2} ( \frac{z_1}{u_1} - \frac{z_j}{u_j} ), $
	for each $ j \in \{ 2,3,\ldots,n \}. $\\
	\noindent Applying the linearity of $ T, $ we have, $ T(\widetilde{z}) = T(z_1,z_2,\ldots,z_n) =  \sum\limits_{j=1}^n \alpha_j T(\widetilde{u_j}). $ Writing the values of $ \alpha_j $ and $ T(\widetilde{u_j}), $ we obtain the following:
	\[ T(z_1,z_2,\ldots,z_n) = l (w_1,w_2,\ldots,w_n), \text{where}~ l = (\frac{3-n}{2} + \frac{1}{2} \sum\limits_{j=2}^n \kappa_j) \frac{z_1}{u_1} + \sum\limits_{j=2}^n (\frac{1}{2}-\frac{1}{2} \kappa_j) \frac{z_j}{u_j}. \]
	Since $ \| \widetilde{w} \|_{1} =1, $ this gives, $ \| T(\widetilde{z}) \|_1 = | l | = | (\frac{3-n}{2} + \frac{1}{2} \sum\limits_{j=2}^n \kappa_j) \frac{z_1}{u_1} + \sum\limits_{j=2}^n (\frac{1}{2}-\frac{1}{2} \kappa_j) \frac{z_j}{u_j} |. $ Let us observe that it is possible to choose unimodular $ z_1,z_2,\ldots,z_n \in \mathbb{C} $ such that $ (\frac{3-n}{2} + \frac{1}{2} \sum\limits_{j=2}^n \kappa_j) \frac{z_1}{u_1} = | \frac{3-n}{2} + \frac{1}{2} \sum\limits_{j=2}^n \kappa_j | $ and $ (\frac{1}{2}-\frac{1}{2} \kappa_j) \frac{z_j}{u_j} = | \frac{1}{2}-\frac{1}{2} \kappa_j | $ for each $ j \in \{ 2,3,\ldots,n \}. $ Therefore, with this particular choice of $ \widetilde{z}, $ we have, 
	
	\begin{align*}
		\| T \|_{\infty,1} & \geq \| T(\widetilde{z}) \|_1\\
		&  =   \frac{1}{2} | (3-n) + \sum\limits_{j=2}^n \kappa_j | + \frac{1}{2}\sum\limits_{j=2}^n | 1-\kappa_j |\\
		& \geq \frac{1}{2} | (3-n)+\sum\limits_{j=2}^n \kappa_j + \sum\limits_{j=2}^n (1-\kappa_j) |\\
		& = \frac{1}{2} | (3-n)+(n-1) |\\
		& = 1.             
	\end{align*}
	
	Since $ \| T \|_{\infty,1} = 1, $ it follows that each inequality in the above system of inequalities must be an equality. Moreover, we note that if  $ | 1 - \kappa_j | = 0 $ for each $ j \in \{ 2,3,\ldots,n \} $ then our claim is already proved. Without loss of generality, let us assume that $ 1 - \kappa_2 \neq 0. $ Therefore, by applying the equality condition of triangular inequality, it follows that there exist scalars $ m_3,m_4,\ldots,m_{n+1} \geq 0 $ such that for each $ j \in \{ 3,4,\ldots,n \}, $ $ 1-\kappa_j = m_j (1-\kappa_2) $ and $ (3-n) + \sum\limits_{j=2}^n \kappa_j = m_{n+1} (1-\kappa_2). $  It follows from these relations that 
	\[ (1-\kappa_2) (\sum\limits_{j=3}^{n+1} m_j + 1) = 2. \]
	
	\noindent Since each $ m_j \geq 0, $ this proves that $ 1-\kappa_2 \in \mathbb{R}, $ i.e., $ \kappa_2 \in \mathbb{R}. $ Consequently, it follows that $ \kappa_j \in \mathbb{R} $ for each $ j \in \{ 3,4,\ldots,n \}. $ This completes the proof of our claim. Moreover, since $ 1=\| T \|_{\infty,1} \geq \| T(\widetilde{u_j}) \| = | \kappa_j |, $ we deduce that $ \kappa_j \in [-1,1] $ for each $ j \in \{ 2,3,\ldots,n \}. $ We further note that since  $ (3-n) + \sum\limits_{j=2}^n \kappa_j = m_{n+1} (1-\kappa_2) \geq 0,  $ we must have, $ \sum\limits_{j=2}^n \kappa_j \geq n-3,  $ i.e., $ \sum\limits_{j=1}^n \kappa_j \geq n-2, $ where $ \kappa_1 = 1. $ This completes the proof of the necessary part of the theorem.
	
	On the other hand, if $ T \in \mathbb{L}(\ell_{\infty}^n, \ell_1^n) $ is such that  $ T $ satisfies the conditions stated in the theorem, then it follows from the above computations that for any $ \widetilde{z} \in S_{\ell_{\infty}^n}, $ we have, $ \| T(\widetilde{z}) \|_1 \leq 1 = \|T(\widetilde{u})\|_1 $ Therefore, $ \| T \|_{\infty,1} =1.  $ This completes the proof of the sufficient part of the theorem and establishes it completely. 
\end{proof}

We apply the above result to obtain a complete characterization of rank $ 1 $ extreme contractions in $ \mathbb{L}(\ell_{\infty}^n, \ell_1^n). $ 

\begin{theorem}\label{theorem:rank $1$ extreme contractions} 
	Let $ T \in \mathbb{L}(\ell_{\infty}^n, \ell_1^n) $ be rank $ 1. $ Then $ T $ is an extreme contraction in $ \mathbb{L}(\ell_{\infty}^n, \ell_1^n) $ if and only if there exists $ \widetilde{u} = (u_1,u_2,\ldots,u_n) \in Ext(B_{l_{\infty}^n}), $ $ \widetilde{w} = (w_1,w_2,\ldots,w_n) \in Ext(B_{\ell_1^n}) $ and $ \kappa_j \in \{\pm 1\}~ (j=2,3,\ldots,n) $ such that for each $ j \in \{ 1,2,\ldots,n \}, $ $ T(\widetilde{u_j}) = \kappa_j \widetilde{w}, $ where $ \kappa_1 = 1 $ and $ \sum\limits_{j=1}^n \kappa_j \geq (n-2).  $
\end{theorem}

\begin{proof}
	We first prove the necessary part of the theorem. Let $ T \in \mathbb{L}(\ell_{\infty}^n, \ell_1^n) $ be a rank $ 1 $ extreme contraction. Clearly, $ \| T \|_{\infty,1} = 1. $ Therefore, it follows from Theorem \ref{theorem:rank $1$} that there exists $ \widetilde{u} = (u_1,u_2,\ldots,u_n) \in Ext(B_{\ell_{\infty}^n}), $ $ \widetilde{w} = (w_1, w_2, \ldots, w_n) \in S_{\ell_1^n} $ and $ \kappa_j \in [-1,1]~ (j=2,3,\ldots,n) $ such that for each $ j \in \{ 1,2,\ldots,n \}, $ $ T(\widetilde{u_j}) = \kappa_j \widetilde{w}, $ where $ \kappa_1 = 1 $ and $ \sum\limits_{j=1}^n \kappa_j \geq (n-2).  $ 
		Our first claim is that $ \widetilde{w} \in Ext(B_{\ell_1^n}). $ Suppose this is not true. Then there exists $ \widetilde{w_1},\widetilde{w_2} \in S_{\ell_1^n} \setminus \widetilde{w} $ and $ t_1 \in (0,1) $ such that $ \widetilde{w} = (1-t_1) \widetilde{w_1} + t_1 \widetilde{w_2}. $ Let us consider two linear operators $ T_1,T_2 \in \mathbb{L}(\ell_{\infty}^n,\ell_{1}^n) $ defined in the following way:
	\[ T_1(\widetilde{u_j}) = \kappa_j \widetilde{w_1} \qquad \qquad \qquad \qquad \qquad T_2(\widetilde{u_j}) = \kappa_j \widetilde{w_2} \quad (j=1,2,\ldots,n). \]
	\\
	Clearly, $ T_1,T_2 \neq T = (1-t_1)T_1 + t_1 T_2. $ Moreover, Theorem \ref{theorem:rank $1$} ensures that $ \|T_1\|_{\infty,1} = \|T_2\|_{\infty,1} = 1. $ However, this contradicts our assumption that $ T $ is an extreme contraction. This completes the proof of our claim that $ \widetilde{w} \in Ext(B_{\ell_{1}^n}). $  Our next claim is that $ \kappa_j \in \{ \pm 1 \} $ for each $ j \in \{ 2,3,\ldots,n \}. $ We first observe that if there exists $ j_0 \in \{ 2,3,\ldots,n \} $ such that $ \kappa_{j_{0}} = -1 $ then for each $ j \in \{ 2,3,\ldots,n \} \setminus \{ j_0 \}, $ we must have that $ \kappa_j = 1. $ We further observe that since $ |\kappa_j| \leq 1 $ and $ \sum\limits_{j=2}^n \kappa_j \geq (n-3),  $ therefore, given any two distinct $ j_1,j_2 \in \{ 2,3,\ldots,n \}, $ we must have that $ \kappa_{j_{1}} + \kappa_{j_{2}} \geq 0. $ Now, if either $ \kappa_2 = -1 $ or $ \kappa_3 = -1 $ then our claim is established. Let us assume that $ \kappa_2,\kappa_3 \neq -1. $ We will show that $ \kappa_2 = \kappa_3 = 1. $ Let us choose $ \epsilon_2 = 1-\kappa_3 $ and $ \epsilon_3 = 1-\kappa_2. $ Clearly, $ \epsilon_2,\epsilon_3 \geq 0. $ Let us consider two linear operators $ A_1,A_2 \in \mathbb{L}(\ell_{\infty}^n,\ell_{1}^n) $ defined in the following way:\\
	
	\begin{align*}
		A_1(\widetilde{u_1}) & = \widetilde{w}    & A_2(\widetilde{u_1})  & =  \widetilde{w}\\
		A_1(\widetilde{u_2}) & = \widetilde{w}     & A_2(\widetilde{u_2}) & =  (\kappa_2 - \epsilon_2)\widetilde{w}\\
		A_1(\widetilde{u_3}) & = (\kappa_3-\epsilon_3)\widetilde{w}     & A_2(\widetilde{u_3}) & = \widetilde{w}\\
		A_1(\widetilde{u_j}) & = \kappa_j \widetilde{w}     & A_2(\widetilde{u_j}) & =  \kappa_j \widetilde{w} \qquad (4 \leq j \leq n).
	\end{align*}
	Since $ |\kappa_2|,|\kappa_3| \leq 1 $ and $ \kappa_2+\kappa_3 \geq 0, $ it follows that $ | \kappa_j - \epsilon_j | \leq 1 ~(j=2,3). $ Moreover, we have, $ 1+1+(\kappa_3-\epsilon_3)+\sum\limits_{j=4}^n \kappa_j = 1+(\kappa_2-\epsilon_2)+1+\sum\limits_{j=4}^n \kappa_j = \sum\limits_{j=1}^n \kappa_j \geq (n-2). $ Therefore, it follows from Theorem \ref{theorem:rank $1$} that $ \| A_1 \|_{\infty,1} = \| A_2 \|_{\infty,1} =1. $ We further observe that $ T = (1-t_0)A_1+t_0 A_2, $ where $ t_0 = \frac{1-\kappa_2}{(1-\kappa_2)+(1-\kappa_3)} \in [0,1]. $ Since $ T $ is an extreme contraction,  we must have $ A_1 = A_2 = T. $ However, this clearly implies that $ \epsilon_2=\epsilon_3=0, $ i.e., $ \kappa_2=\kappa_3=1. $ Using similar arguments, it can be shown that $ \kappa_j \in \{ \pm 1 \} $ for each $ j \in \{ 2,3,\ldots,n \}. $ This completes the proof of the necessary part of the theorem.
	
	Let us now prove the comparatively easier sufficient part of the theorem. Let $ T \in \mathbb{L}(\ell_{\infty}^n,\ell_{1}^n) $ satisfies all the conditions stated in the theorem. Theorem \ref{theorem:rank $1$} ensures that $ \| T \|_{\infty,1} = 1. $ Suppose there exists $ t \in [0,1] $ and $ A_3,A_4 \in \mathbb{L}(\ell_{\infty}^n,\ell_{1}^n), $ with $ \| A_3 \|_{\infty,1} = \| A_4 \|_{\infty,1} =1, $ such that $ T = (1-t)A_3+tA_4. $ Since $ \pm \widetilde{w} \in Ext(B_{\ell_{1}^n}), $ it is easy to show that $ A_3=A_4=T. $ This proves that $ T $ is an extreme contraction in  $ \mathbb{L}(\ell_{\infty}^n,\ell_{1}^n). $ This completes the proof of the sufficient part of the theorem and establishes it completely.  
\end{proof}

We next prove that the pair $(\ell_{\infty}^n,\ell_{1}^n)$ satisfies the rank invariant Krein-Milman Property of order $1,$ as promised. 

\begin{theorem}\label{theorem:convex combination}
	The pair $ (\ell_{\infty}^n,\ell_{1}^n) $ satisfies the rank invariant Krein-Milman Property of order $ 1, $ for every $ n \in \mathbb{N}. $
\end{theorem}

\begin{proof}
	Let $ T $ be a rank $ 1 $ linear operator of unit norm in $ \mathbb{L}(\ell_{\infty}^n,\ell_{1}^n). $ Let $ T(\widetilde{u})= \widetilde{w} \in S_{l_1^n}$ and  for any $ 2 \leq j \leq n ,$ let $ T(\widetilde{u_{j}}) = k_j \widetilde{w},$   where $ \kappa_1 = 1, k_j \in [-1, 1]$ and $ \sum\limits_{j=1}^n \kappa_j \geq (n-2).  $   Suppose that $ \widetilde{w}= c_1 \widetilde{w_1} + c_2 \widetilde{w_2} + \ldots + c_r \widetilde{w_r},$ for some $r \in \mathbb{N},$ where $ \widetilde{w_1}, \widetilde{w_2}, \ldots, \widetilde{w_r} \in Ext(B_{\ell_{1}^n})$ and $ c_i \in [0, 1],$ for any $ 1 \leq i \leq r$ and  $ c_1 + c_2 + \ldots + c_r =1.$ Now let us assume that for any $ 1 \leq i \leq r,$ 
	$ T_i \in \mathbb{L}(\ell_{\infty}^n,\ell_{1}^n)$ such that for any $ 2 \leq j \leq n,$
	\begin{eqnarray*}
		 T_i(\widetilde{u}) &=& \widetilde{w_i} \\
		T_i(\widetilde{u_{j}}) &=& k_j\widetilde{w_i}. 
	\end{eqnarray*}
	   Clearly, $ T = c_1 T_1 + c_2 T_2 + \ldots + c_r T_r.$
 	We employ the same arguments as given in the proof of the necessary part of Theorem \ref{theorem:rank $1$ extreme contractions} to express each  $ T_i $ as a convex combination of two rank $ 1 $ operators $ A_{i1} $ and $ A_{i2} $ of unit norm in $ \mathbb{L}(\ell_{\infty}^n,\ell_{1}^n). $ We notice that $ A_{i1} $ and $ A_{i2} $ are ``one step closer" to being rank $ 1 $ extreme contractions in $ \mathbb{L}(\ell_{\infty}^n,\ell_{1}^n), $ in comparison to each $ T_i. $ We continue applying the same technique on $ A_{i1}, A_{i2}, $ and the subsequent operators obtained in the process until we arrive at rank $ 1 $ extreme contractions. We note that that process would terminate in at most $ (n-1) $ steps. This gives us the required expression of $ T $ as a convex combination of rank $ 1 $ extreme contractions in $ \mathbb{L}(\ell_{\infty}^n,\ell_{1}^n) $ and completes the proof of the theorem. 
\end{proof}

We end this section with the following remark, explaining the motivation behind our choice of the particular pair $ (\ell_{\infty}^n,\ell_{1}^n). $

\begin{remark}
	The Grothendieck's inequality, one of the most fundamental and celebrated results in the metric theory of tensor products of Banach spaces, is closely related to the extremal structure of the unit ball of the space $ \mathbb{L}(\ell_{\infty}^n,\ell_{1}^n), $ considered over the complex field. We refer the readers to the remarkably resourceful and influential article \cite{P}, for more information on this rapidly advancing area of research. In particular, a complete description of the extreme contractions of all rank in the space $ \mathbb{L}(\ell_{\infty}^n,\ell_{1}^n) $ might be useful in estimating the Grothendieck's constant. This explains our choice of the particular pair $  (\ell_{\infty}^n,\ell_{1}^n). $ In this context, Theorem \ref{theorem:rank $1$ extreme contractions} allows us to identify the \emph{rank one} extreme contractions, and Theorem \ref{theorem:convex combination} ensures that we have a complete description of the \emph{rank one} linear operators in the said space. We would like to point out that the descriptions of  higher rank extreme contractions seem to be far more difficult, and the problem remains open to the best of our knowledge.
\end{remark}

\section*{Section-III}

This section is devoted to studying the extreme contractions on finite-dimensional polyhedral Banach spaces. In order to characterize the extreme contractions, we need to introduce some new definitions, motivated by the geometry of polyhedral Banach spaces.

	Let $\mathbb{X}$ be a finite-dimensional polyhedral Banach space and let $v \in S_{\mathbb{X}}.$ Observe that  if there exists an $i$-face $F$ of $B_{\mathbb{X}}$ such that $v \in F,$ for some $1 \leq i < n-1,$ then there exists an $(i+1)$-face of $B_{\mathbb{X}}$ containing $v.$ If $F$ is an  $i$-face containing $v$  such that there exists no $(i-1)$-face of $B_{\mathbb{X}}$ containing $v,$  then $F$ is said to be the minimal face of $v.$ Note that the minimal face of an element $v \in S_{\mathbb{X}}$ is always unique. Next we introduce the following definitions whose importance in describing the extreme contractions between finite-dimensional polyhedral Banach spaces will be clear in due course of time.

\begin{definition}
	Let $\mathbb{X}$ and $\mathbb{Y}$ be finite-dimensional polyhedral Banach spaces and let $T \in S_{\mathbb{L}(\mathbb{X}, \mathbb{Y})}.$ If $S$ is a maximal (with respect to the usual partial order induced by set inclusion) linearly independent  subset of $M_T \cap Ext(B_{\mathbb{X}})$ such that $T(S) \subset Ext(B_{\mathbb{Y}})$ then the cardinality of $S$ is defined as \emph{the extremal number corresponding to $T.$}
\end{definition}

It is elementary to observe that the extremal number corresponding to $ T $ is independent of the choice of $ S $ in the above definition.

\begin{definition}
	Let $\mathbb{X}$  be a finite-dimensional polyhedral Banach space. Let $v \in S_{\mathbb{X}}$ and let $F$ be the minimal face of $ v.$ Suppose that  $Ext(B_{\mathbb{X}}) \cap F=\{ w_1, w_2, \ldots, w_r\}.$ Let $\epsilon > 0.$ An $r$-tuple of scalars $( \mu_1, \mu_2, \ldots, \mu_r)$ is said to be \emph{an $F$-associated tuple of scalars corresponding to $v$ with $\epsilon$-dominated norm} if $ v+ \sum_{i=1}^{r} \mu_i w_i , ~ v- \sum_{i=1}^{r} \mu_i w_i \in F$ and $\| \sum_{i=1}^{r} \mu_i w_i \| < \epsilon.$  
\end{definition}

By applying the convexity property of the norm function, it is not difficult to observe that if $(\mu_1, \mu_2, \ldots, \mu_r)$ is an $F$-associated tuple of scalars corresponding to $v$ with $\epsilon$-dominated norm then $ \sum_{i=1}^{r} \mu_i =0.$

\begin{definition}
	Let $\mathbb{X}$ be a finite-dimensional polyhedral Banach space. Let $ x \in B_{\mathbb{X}}$ and let $\epsilon > 0.$ Then a pair of elements $u, v \in \mathbb{X}$ is said to be  \emph{a pair of $\epsilon$-diametric points of $x$} if  $ u, v \in \mathcal{B}(x, \epsilon) \cap B_{\mathbb{X}}$ and  $u+v=2x.$
\end{definition}

\begin{definition} 
	Let $\mathbb{X}$ and $\mathbb{Y}$ be finite-dimensional polyhedral Banach spaces, where $dim(\mathbb{X})=n.$ Let  $T \in S_{\mathbb{L}(\mathbb{X}, \mathbb{Y})}$ and let $(x,y) \in S_{\mathbb{X}} \times S_{\mathbb{Y}}.$ Suppose that $F$ is the minimal face of $y \in S_{\mathbb{Y}}$ such that $ Ext(B_{\mathbb{Y}}) \cap F= \{ w_1, w_2, \ldots, w_r\}.$ Also assume that $ S =\{ x, z_2, z_3, \ldots, z_n\} \subset S_{\mathbb{X}} $ is a linearly independent set. 
	Then $(x, y)$ is called \emph{a $*$-pair of points with respect to $ S $}  
	if given any $\epsilon> 0$ and any pair of $\epsilon$-diametric points $u_i, v_i$ of $Tz_i,$ where $ 2 \leq i \leq n,$ there exists an $\widetilde{v}= \alpha_1 x+ \sum_{i=2}^{n} \alpha_i z_i \in Ext(B_{\mathbb{X}})$ such that the following condition holds:
	\[ \| \alpha_1 y+ \sum_{i=2}^n \alpha_i u_i + \alpha_1\bigg( \sum_{j=1}^{r} \mu_j w_j \bigg) \| > 1	\]	or
	\[ \| \alpha_1 y+ \sum_{i=2}^n \alpha_i v_i - \alpha_1 \bigg( \sum_{j=1}^{r} \mu_j w_j \bigg) \| > 1,	\]
	for any  $F$-associated tuple of scalars corresponding to $y$ with $\epsilon$-dominated norm, $(\mu_1, \mu_2, \ldots, \mu_r),$ where  $ \alpha_1 (\sum_{j=1}^{r} \mu_j w_j ) \neq \theta $ whenever $ (u_2, u_3, \ldots, u_n)= (v_2, v_3, \ldots, v_n ).$
\end{definition}

The following result describes a nice property of the faces of a polyhedral Banach spaces.

\begin{proposition}\label{minimal face}
	Let $\mathbb{X}$ be a finite-dimensional polyhedral Banach space and let $v \in S_{\mathbb{X}}.$ Suppose that $F$ is a face of $ B_{\mathbb{X}}, $ containing $v.$ Then for any $ u, w \in S_{\mathbb{X}}$ such that $v = (1-t)u+ tw,$ it follows that $u, w \in F.$
\end{proposition}
\begin{proof}
Since $F$ is a face of $v \in S_{\mathbb{X}},$ it can be written as $F = S_{\mathbb{X}} \cap \delta M,$ where  $\delta M$ is the boundary of a  closed half-space $M$ in $ \mathbb{X}. $ It is easy to observe that there exists a  functional $f \in S_{\mathbb{X}^*}$ such that $\delta M= \{ x \in \mathbb{X}: f(x)=1\}.$ Now, $f((1-t)u+tw)= f(v)=1$ and from this, it is easy to see that $f(u)= f(w)=1,$ by using that $ \| f \| = 1. $  Since $u, w\in S_{\mathbb{X}},$  we conclude that $u, w \in F.$ 
\end{proof}


We are  now ready to present a complete characterization of  the extreme contractions of  $\mathbb{L}(\mathbb{X}, \mathbb{Y}),$ in terms of our newly introduced notion of $*$-pair of points, where $\mathbb{X}, \mathbb{Y}$ are finite-dimensional polyhedral Banach spaces.

\begin{theorem}\label{characterization}
	Let $\mathbb{X}$ and $\mathbb{Y}$ be  finite-dimensional polyhedral Banach spaces and let $dim(\mathbb{X})=n.$  $T \in S_{\mathbb{L}(\mathbb{X}, \mathbb{Y})}$  is an extreme contraction if and only if exactly one of the following conditions holds true:\\
	(i) Extremal number of $T$ is equal to $n$.\\
	(ii) Extremal number of $T$ is strictly less than $n$ and there exists $v \in M_T \cap Ext(B_{\mathbb{X}})$ and $ S \subset S_{\mathbb{X}}$, a basis of $\mathbb{X}$ containing $v$ such that   $Tv \notin Ext(B_{\mathbb{Y}})$ and $(v, Tv) $ is a $*$-pair of points with respect to $ S.$
\end{theorem}

\begin{proof}
	Let us first prove the sufficient part of the theorem. Assume that the  extremal number of $T$ is equal to $n.$ Suppose on the contrary that $T= \frac{1}{2}T_1+\frac{1}{2} T_2,$ for some $T_1, T_2 \in S_{\mathbb{L}(\mathbb{X}, \mathbb{Y})}.$  Suppose  $v_1, v_2, \ldots, v_n $ are linearly independent points in $M_{T}\cap Ext(B_{\mathbb{X}})$ such that $Tv_i \in Ext(B_{\mathbb{Y}}),$ for all $1 \leq i \leq n.$ Then $Tv_i= \frac{1}{2}T_1v_i+\frac{1}{2}T_2v_i,$ for all $1 \leq i \leq n.$ Since $Tv_i \in Ext(B_{\mathbb{Y}}),$ we get  $Tv_i= T_1v_i= T_2v_i,$ for all $1 \leq i \leq n.$ Therefore, $T= T_1= T_2.$  This implies that $T$ is an extreme contraction.
	Next assume that the extremal number of $T$ is strictly less than $n$ and that there exists  $v \in M_T \cap Ext(B_{\mathbb{X}})$ such that $ Tv \notin Ext(B_{\mathbb{Y}}).$  Also, let     $ S = \{ v , z_2, z_3, \ldots, z_n\} \subset S_{\mathbb{X}} $ be a basis of $\mathbb{X}$ containing $v$ such that $(v,Tv) $ is a $*$-pair of points with respect to $S.$  Suppose on the contrary that $T$ is not an extreme contraction. Taking $\epsilon> 0,$ it is easy to observe that there exists $T_1, T_2 \in S_{\mathbb{L}(\mathbb{X}, \mathbb{Y})}$ such that $T= \frac{1}{2} T_1+ \frac{1}{2} T_2$ and $\|T_1-T \| < \epsilon, \|T_2-T \|< \epsilon.$   Let $F$ be the minimal face  of the element $Tv$ and  $F \cap Ext(B_{\mathbb{Y}})= \{ w_1, w_2, \ldots, w_r\}.$ Therefore, $Tv = c_1w_1+ c_2w_2+ \ldots+ c_rw_r,$ for some $c_i$'s $ \in (0, 1)$ such that  $ \sum_{i=1}^{r} c_i= 1.$ 
	Since $v \in M_T,$ it easily follows that $v \in M_{T_1}\cap M_{T_2}.$  Also, as $ Tv = \frac{1}{2}( T_1v + T_2v) $ and $F$ is the minimal face of $Tv,$ we conclude that  $T_1v, T_2v \in F,$  by Proposition \ref{minimal face}.
	Therefore,
	\begin{eqnarray*}
		T_1v=  \sum_{i=1}^r (c_i+\mu_i)w_i= Tv +\sum_{i=1}^r \mu_iw_i,\\
		T_2v=  \sum_{i=1}^r (c_i-\mu_i)w_i= Tv -\sum_{i=1}^r \mu_iw_i,\\
	\end{eqnarray*}
	for some   $F$-associated tuple  of scalars,  $ (\mu_1,\mu_2, \ldots, \mu_r),$ corresponding to $Tv$ with $\epsilon$-dominated norm.
	It is easy to verify that  $  T_1 z_i$ and $  T_2z_i$ is a  pair of $\epsilon$-diametric points of each $Tz_i,$ for any $ 2 \leq i \leq n.$
	Let  $\widetilde{x} = \alpha_1v+\sum_{i=2}^{n}\alpha_i z_i, \in Ext(B_{\mathbb{X}}),$  for some scalars $ \alpha_i$'s. Then clearly $  \|T_1\widetilde{x} \| \leq 1$  and $\|T_2\widetilde{x}\| \leq 1.$
	Therefore, 
	\begin{eqnarray*}
		\|  \alpha_1 Tv + \alpha_1 \bigg( \sum_{j=1}^r \mu_jw_j \bigg) + \sum_{i=2}^{n}\alpha_i  T_1z_i \| \leq 1,\\
		\| \alpha_1 Tv - \alpha_1  \bigg( \sum_{j=1}^r \mu_jw_j \bigg) + \sum_{i=2}^{n} \alpha_i  T_2z_i \| \leq 1. 
	\end{eqnarray*}
	Summarizing,  there exists  $\epsilon > 0$  and  a  pair of $\epsilon $-diametric points $ T_1z_i, T_2z_i$  of $Tz_i,$  ($ 2 \leq i \leq n $)  such that for any $\widetilde{x}= \alpha_1v+\sum_{i=2}^{n}\alpha_i z_i \in Ext(B_{\mathbb{X}}), $  we have that
	\begin{eqnarray*}
		\|  \alpha_1 Tv + \alpha_1  \bigg( \sum_{j=1}^r \mu_jw_j \bigg) + \sum_{i=2}^{n}\alpha_i T_1z_i \| \leq 1 \\
		\mbox{and}   \, \| \alpha_1 Tv - \alpha_1  \bigg( \sum_{j=1}^r \mu_jw_j \bigg) + \sum_{i=2}^{n}\alpha_i T_2z_i \| \leq 1,
	\end{eqnarray*}
	for some $F$-associated tuple of scalars,  $ (\mu_1, \mu_2, \ldots, \mu_r)$ corresponding to $Tv$ with $\epsilon$-dominated norm. Since $T \neq T_1 \neq T_2,$ so  $ \alpha_1  ( \sum_{j=1}^r \mu_jw_j ) \neq \theta $ whenever $(T_1z_2, T_1z_3, $ $ \ldots,  T_1z_n)= (T_2z_2, T_2z_3, \ldots, T_2z_n).$ 
	This contradicts the fact  that  $ (v,Tv)$ is an $*$-pair of points with respect to $ S.$
	
	\medskip
	
	Next we  prove the necessary part of the theorem. Let $ T \in S_{\mathbb{L} (\mathbb{X}, \mathbb{Y})} $ be an extreme contraction. From \cite[Th. 2.2]{SRP},  we get that $ span(M_T \cap Ext(B_{\mathbb{X}}))= \mathbb{X}.$ If the extremal number of $T$ is  equal to $n,$ then we have nothing to prove. Let us now consider that the extremal number of $ T $ is strictly less than $n.$ Let $v \in M_T \cap Ext(B_{\mathbb{X}})$ such that $Tv \notin Ext(B_{\mathbb{Y}}).$ Let $F$ be the minimal face of $Tv$ and let $Ext(B_{\mathbb{Y}}) \cap F= \{ w_1, w_2, \ldots, w_r\}.$ Suppose on the contrary that $(v,Tv)$ is not a $*$-pair of points with respect to $ S,$ for any basis $ S \subset S_{\mathbb{X}}$ of $\mathbb{X}$ containing $v$. Let $ S =\{ v, z_2, z_3, \ldots, z_n\} \subset S_{\mathbb{X}}.$
	Then  there exists  $\epsilon > 0$ and a pair of $\epsilon$-diametric points $u_i, v_i$ of  $ Tz_i $ $  ( 2 \leq i \leq n) $ such that for any $\widetilde{x}= \alpha_1 x+ \sum_{i=2}^{n} \alpha_i z_i \in Ext(B_{\mathbb{X}}),$  we have
	\begin{eqnarray*}
		\| \alpha_1 Tv+ \sum_{i=2}^n \alpha_i u_i + \alpha_1  \bigg( \sum_{j=1}^r \mu_jw_j \bigg) \| \leq  1 \\
		\| \alpha_1 Tv+ \sum_{i=2}^n \alpha_i v_i - \alpha_1  \bigg( \sum_{j=1}^r \mu_jw_j \bigg) \| \leq 1,
	\end{eqnarray*}
	for some   $F$-associated tuple of scalars,  $ (\mu_1, \mu_2, \ldots, \mu_r)$ corresponding to $Tv$ with $\epsilon$-dominated norm  such that  $ \alpha_1  ( \sum_{j=1}^r \mu_jw_j ) \neq  \theta $ whenever $ (u_2, u_3, \ldots, u_n)= (v_2, v_3, $ $ \ldots,  v_n).$ 
	Define $T_1, T_2 \in L(\mathbb{X}, \mathbb{Y})$  as 
	\begin{eqnarray*}
		T_1v &=& Tv +  \sum_{j=1}^r \mu_jw_j \\
		T_1z_i &=& u_i, \quad 2 \leq i \leq n\\
		\mbox{and}~	T_2v &=& Tv -   \sum_{j=1}^r \mu_jw_j \\
		T_2z_i &=& v_i,  \quad  2 \leq i \leq n.
	\end{eqnarray*} 
	It is now immediate that $ T= \frac{1}{2} T_1 + \frac{1}{2} T_2.$ Since $ \alpha_1  ( \sum_{j=1}^r \mu_jw_j ) \neq     \theta $ whenever $ (u_2, u_3, \ldots, u_n)= (v_2, v_3, \ldots, v_n),$  we have $T_1 \neq T \neq T_2.$
	For any $ \widetilde{x} \in Ext(B_{\mathbb{X}}),$ 
	\begin{eqnarray*}
		\| T_1\widetilde{x} \| &=& \| T_1(  \alpha_1 v+ \sum_{i=2}^{n} \alpha_i z_i ) \|\\ 
		& = & \|  \alpha_1 Tv + \sum_{i=2}^{n} \alpha_i  u_i  + \alpha_1  \bigg( \sum_{j=1}^r \mu_jw_j \bigg) \|  \leq 1.
	\end{eqnarray*} 
	Similarly, $ \| T_2\widetilde{x} \| \leq 1.$ Thus, $T_1, T_2 \in B_{\mathbb{L}(\mathbb{X},\mathbb{Y})},$ and this  contradicts  the assumption that $T$ is an extreme contraction. This establishes the theorem.
\end{proof}

\begin{remark}
	From the above theorem, it is easy to observe that for the pair $(\mathbb{X}, \mathbb{Y})$ such that there exists $T \in Ext(\mathbb{L}(\mathbb{X}, \mathbb{Y})) $ satisfying $T(Ext(B_\mathbb{X})) \cap Ext(B_{\mathbb{Y}}) \neq \phi,$ we  obtain a $v \in Ext(B_{\mathbb{X}})$ and  a basis $S\subset S_{\mathbb{X}}$ of $ \mathbb{X}$ containing $v$ such that $(v, Tv)$ is a $*$-pair of points with respect to $S$, for some $T \in S_{\mathbb{L}(\mathbb{X}, \mathbb{Y})}.$
\end{remark}

The following result, which also illustrates the application of $*$-pair of points in the study of estreme contractions, follows as a consequence of the above theorem.

\begin{corollary}\label{example:corollary}
	Let $ \mathbb{X}$ be a $2$-dimensional polyhedral Banach space such that $ $ $ | Ext(B_{\mathbb{X}}) | \geq 8.$ Then any linear operator $ T \in S_{\mathbb{L}(\mathbb{X}, \ell_{\infty}^2)} $ such that $ | M_T \cap Ext(B_{\mathbb{X}}) | =
	8$ is an extreme contraction.
\end{corollary}

\begin{proof}
	Let $ \pm F_1, \pm F_2$ be the edges of $B_{l_{\infty}^2}$, where $ \{(1,1), (1,-1) \}$ and $ \{(1,1), (-1,1) \}$ are the extreme points of the edges $ F_1$ and $F_2,$ respectively.  Suppose that  $ M_T \cap Ext(B_{\mathbb{X}})= \{ \pm v_1, \pm v_2, \pm v_3, \pm v_4\}.$ It is easy to verify that   $| T(Ext(B_{\mathbb{X}})) \cap F_i| =2,$ where $i \in \{1,2\}.$
	Without loss of generality, let us assume that $ Tv_1, Tv_2 \in F_1$ and $ Tv_3, Tv_4 \in F_2,$ where $ Tv_i \notin Ext (B_{l_{\infty}^2}),$ for any $i \in \{1, 2, 3, 4\}.$ 
	Let $\epsilon > 0.$ Since $ Tv_3 \in F_2, $ it is easy to observe from Proposition \ref{minimal face} that for any pair of $\epsilon$-diametric points  $w_1, w_2$ of $Tv_3,$ it holds that  $ w_1, w_2 \in F_2$.
	Moreover,  $w_1, w_2$  can be written as 
	\begin{eqnarray*}
		w_1= Tv_3-\gamma[(1,1)- (-1,1)]= Tv_3 - \gamma(2, 0), \\
		w_2= Tv_3+\gamma[(1,1)- (-1,1)]= Tv_3 + \gamma(2, 0),
	\end{eqnarray*}
	for some $ \gamma \in [0, \frac{\epsilon}{2}).$ 
	Clearly, $F_1$ is the minimal face of $Tv_1.$ Suppose that $( \mu_1, \mu_2)$ is an $ F_1$-associated tuple of scalars corresponding to  $Tv_1$ with $\epsilon$-dominated norm such that  $a_1(\mu_1(1, 1)+ \mu_2(1, -1)) \neq \theta,$ whenever $w_1=w_2.$    It is now easy to observe that $\mu_1+ \mu_2=0$ and at least one of  $ \mu_1 $ and $\gamma$ is non-zero.

	 Let $ v_2= a_1v_1+ b_1v_3$ and let $ Tv_2 = ( 1, \alpha),$ where $|\alpha|<1,$ then  
	\begin{eqnarray*}
		\| a_1Tv_1 &+& a_1\{\mu_1(1,1)+ \mu_2(1,-1)\}+ b_1 w_1 \| \\ 
		&=&
		\| a_1Tv_1+ b_1 Tv_3 + a_1 \mu_1\{(1,1)-(1,-1)\} - b_1 \gamma(2,0) \|\\ 
		& =& \| Tv_2 + a_1\mu_1(0,2) - b_1\gamma(2,0) \|\\ 
		& =& \| (1, \alpha) + a_1\mu_1(0,2) - b_1\gamma(2,0)\|\\ 
		& =& \| ( 1- 2b_1 \gamma , \alpha + 2a_1\mu_1) \|
	\end{eqnarray*}
	and similarly, 
	\[
	\| a_1Tv_1 - a_1\{\mu_1(1,1)+ \mu_2(1,-1)) \}+ b_1 w_2 \|= \| (1+ 2b_1 \gamma, \alpha- 2a_1\mu_1) \|. \]
		Again, let $v_4= a_2v_1 + b_2v_3 $ and $ Tv_4 = (\beta, 1),$ where $|\beta|<1.$ Then 
	\begin{eqnarray*}
		\| a_2Tv_1 + a_2\{\mu_1(1,1)+ \mu_2(1,-1)\}+ b_2 w_1 \| & = & \| a_2Tv_1 + a_2\mu_1(0, 2) + b_2 w_1 \| \\
		& = & \| Tv_4 + a_2 \mu_1 (0, 2) - b_2 \gamma ( 2, 0) \| \\ 
		& =& \| ( \beta, 1) + a_2 \mu_1 ( 0, 2) - b_2 \gamma ( 2, 0) \| \\
		& = & \| (\beta - 2 b_2 \gamma, 1+ 2 a_2 \mu_1 ) \|.		 
	\end{eqnarray*} 
	and \[
	\| a_2Tv_1 - a_2\{\mu_1(1,1)+ \mu_2(1,-1)\}+ b_2 w_2 \| =  \| (\beta + 2 b_2 \gamma, 1- 2 a_2 \mu_1 ) \|.
	\]
	Since at least $\gamma$ and $\mu_1$ is non-zero, it is immediate that at least one of  $ |1- 2b_1 \gamma|, |1+ 2b_1 \gamma|, | 1+ 2 a_2 \mu_1|$ and $ | 1 - 2 a_2 \mu_1|$ is strictly greater than $1$ 
	and accordingly, at least one of the following must holds: 
	\begin{eqnarray*}
		\| a_1Tv_1 + a_1\{\mu_1(1,1)+ \mu_2(1,-1)\}+ b_1 w_1 \| >1\\ 
		\| a_1Tv_1 - a_1\{\mu_1(1,1)+ \mu_2(1,-1)\}+ b_1 w_2\| > 1\\
		\| a_2Tv_1 + a_2\{\mu_1(1,1)+ \mu_2(1,-1)\}+ b_2 w_1\| >1 \\
		\| a_2Tv_1 - a_2\{\mu_1(1,1)+ \mu_2(1,-1)\}+ b_2 w_2\| > 1.
	\end{eqnarray*}
	In other words, there exists $ S=\{ v_1, v_3\} \subset S_{\mathbb{X}}$ such that $Tv_1 \notin Ext( B_{\ell_{\infty}^2}) $ and $ ( v_1, Tv_1)$ is a $*$-pair of points with respect to $S$. From Theorem \ref{characterization}, we conclude that $T$ is an extreme contraction.
\end{proof}

\begin{remark}
	This result substantiates \cite[Th. 2.8]{MPD}. We observe that any linear operator with  $ | M_T \cap Ext(B_{\mathbb{X}}) | =	8$ is an extreme contraction but $T(M_T \cap Ext(B_{\mathbb{X}})) \cap Ext (B_{\ell_{\infty}^2}) = \phi.$ 
\end{remark}

\noindent \textbf{Acknowledgement.} \\
The authors would like to thank the anonymous referee for his invaluable suugestions and comments. The research of Dr. Debmalya Sain is sponsored by a Maria Zambrano postdoctoral grant. Dr. Sain feels elated to acknowledge the pleasant friendship of Mr. Mrinal Jana in his life! The second  author would like to thank  CSIR, Govt. of India, for the financial support in the form of Junior Research Fellowship under the mentorship of Prof. Kallol Paul.


\begin{thebibliography}{99}
	

\bibitem[1]{A}  Alexandrov, A, D., \textit{Convex polyhedra}, Springer Berlin Heidelberg New York (2005).

	\bibitem[2]{B} Birkhoff, G.,  \textit{Orthogonality in linear metric spaces}, Duke Math. J. \textbf{1} (1935), 169--172.

	\bibitem[3]{G1} Grz\c{a}\'slewicz, R.,  \textit{Extreme operators on $2$-dimensional $l_p$-spaces}, Colloq. Math., \textbf{44} (1981) 309-315.
	
	\bibitem[4]{G2} Grz\c{a}\'slewicz, R.,  \textit{Extreme operators on real Hilbert spaces}, Math. Ann., \textbf{261} (1982) 463-466. 
	
		\bibitem[5]{GKT} Gale. D.,  Kuhn. H. W., and Tucker. A. W.,\textit{ Linear programming and the theory of games}, Activity Analysis of Production and Allocation, (1951), 317-329.
	
	\bibitem[6]{I} Iwanik, A., \textit{Extreme contractions on certain function spaces}, Colloq. Math.,\textbf{40} (1978) 147-153.
	
	\bibitem[7]{J} James, R. C., \textit{Orthogonality and linear functionals in normed linear spaces}, Trans.  American Math. Soc., \textbf{61} (1947b), 265-292.

\bibitem[8]{K1} Kim, C. W., \textit{Extreme contraction operators on $l_\infty$}, Math. Z., \textbf{151} (1976) 101-110.

\bibitem[9]{K2} Kadison, R. V., \textit{Isometries of operator algebras}, Ann. of Math.,\textbf{54} (1951), 325-338.
	
	\bibitem[10]{L1} Lima, \r A., \textit{ On extreme operators on finite-dimensional Banach spaces whose unit balls are polytypes}, Ark. Mat., \textbf{19} (1981), no. 1, 97-116.
	
\bibitem[11]{L2}	 Lima, \r A., \textit{ Intersection properties of balls in spaces of compact operators}, Ann. Inst. Fourier, 	\textbf{28} (1978) 35-65.

\bibitem[12]{L3} Lima, \r A., and  Olsen, G., \textit{ Extreme points in duals of complex operator spaces}, Proc. Amer. Math. Soc., \textbf{94} (1985) no. 3, 437-440.

	
	
	\bibitem[13]{LP} Lindenstrauss, J., Perles, M. A., \textit{On extreme operators in finite-dimensional spaces}, Duke Math, J., \textbf{36} (1969), 301-314. 
		
	\bibitem[14]{MPD} Mal, A., Paul, K., Dey, S.,  \textit{Characterization of extreme contractions through k-smoothness of operators},  Linear Multilinear Algebra, (2021) DOI: 
	 10.1080/03081087.2021.1913086.
	 
	 
	 
	 \bibitem[15]{P} Pisier, G., \textit{Grothendieck's Theorem, past and present}, Bull. Amer. Math. Soc., \textbf{49}(2) (2012), 237-323.
	 
	 
	 
	 
	 
	 
	  
	 \bibitem[16]{RRBS}   Ray. A., Roy. S., Bagchi. S., Sain. D., \textit{ Extreme contractions on finite-dimensional polygonal Banach spaces-II}, J. Operator Theory, \textbf{84} (2020), 127-137.
	 
	 
	 	\bibitem[17]{RS} Roy, S., Sain, D.,  \textit{Numerical radius and notion of smoothness in the space of bounded linear operators}, Bull. Sci. Math., \textbf{173} (2021), 103070.
	 


	\bibitem[18]{SPBB} Sain, D.,  Paul, K., Bhunia, P., Bag, Santanu., \textit{On the numerical index of polyhedral Banach space}, Linear Algebra Appl., \textbf{577} (2019), 121-133.
	
	\bibitem[19]{SPM}   Sain, D.,  Paul, K., Mal, A., \textit{On extreme contractions between real Banach spaces}, Expo. Math., (2020), https://doi.org/10.1016/j.exmath.2019.09.004.
	


\bibitem[20]{SRP} Sain, D., Ray, A., Paul, K., \textit{Extreme contractions on finite-dimensional polygonal Banach spaces}, J.Conv. Anal., \textbf{26} (2019), 877-885.


\bibitem[21]{S1} Sharir, M., \textit{Characterization and properties of extreme operators into $C(Y)$}, Israel J. Math., \textbf{12} (1972), 174-183.

\bibitem[22]{S2} Sharir, M., \textit{Extremal stucture in operator spaces}, Trans. Amer. Math. Soc., \textbf{186} (1973), 91-111.


	
	
	

	
	

	  
	 
	  
	  
	
	
\end{thebibliography}
\end{document}